\documentclass[11pt]{article}

 \usepackage{amsfonts,amssymb,amsmath,mathrsfs,mathtools,pifont}
\usepackage{bm,bbm}
\usepackage{color,latexsym,amsfonts}
\usepackage{graphicx}
\usepackage{caption}
\usepackage{lipsum}

\newcommand{\qed}{\hfill\rule{2mm}{3mm}\vspace{4mm}}

\numberwithin{equation}{section}

\allowdisplaybreaks[3]

 \newtheorem{theorem}{Theorem}[section]
 \newtheorem{lemma}[theorem]{Lemma}
 
 \newtheorem{remark}{Remark}

 \def\<{\langle}\def\>{\rangle}

 \def\beqnn{\begin{eqnarray*}}\def\eeqnn{\end{eqnarray*}}
 \def\<{\langle}\def\>{\rangle}

\def\beqlb{\begin{eqnarray}}\def\eeqlb{\end{eqnarray}}
 \def\qed{\hfill$\Box$\medskip}

\begin{document}

\bigskip\bigskip
\noindent{\Large\bf Fixed points with  finite mean of the smoothing transform in random environments }\footnote{
\noindent  Supported by NSFC (NO.11531001).
}

\noindent{%\normalsize\sf
Wenming Hong\footnote{ School of Mathematical Sciences
\& Laboratory of Mathematics and Complex Systems, Beijing Normal
University, Beijing 100875, P.R. China. Email: wmhong@bnu.edu.cn} ~ Xiaoyue Zhang\footnote{ School of Mathematical Sciences
\& Laboratory of Mathematics and Complex Systems, Beijing Normal
University, Beijing 100875, P.R. China. Email: zhangxiaoyue@mail.bnu.edu.cn}

\noindent{%\normalsize\sf
(Beijing Normal University)
}

}

\begin{center}
\begin{minipage}{12cm}
\begin{center}\textbf{Abstract}\end{center}
\footnotesize
At each time $n\in\mathbb{N}$, let $\bar{Y}^{(n)}=(y_{1}^{(n)},y_{2}^{(n)},\cdots)$ be a random sequence of non-negative numbers that are ultimately zero in a random environment $\xi=(\xi_{n})_{n\in\mathbb{N}}$ in time, which satisfies for each $n\in\mathbb{N}$ and a.e. $\xi,~E_{\xi}[\sum_{i\in\mathbb{N}_{+}}y_{i}^{(n)}(\xi)]=1.$ The existence and uniqueness of the non-negative fixed points of the associated smoothing transform in random environments is considered. These fixed points are solutions of the distributional equation for $a.e.~\xi,~Z(\xi)\overset{d}{=}\sum_{i\in\mathbb{N}_{+}}y_{i}^{(0)}(\xi)Z_{i}(T\xi),$
where when given the environment $\xi$, $Z_{i}(T\xi)~(i\in\mathbb{N}_{+})$ are $i.i.d.$ non-negative random variables, and distributed the same as $Z(\xi)$.
 As an application, the martingale convergence of the branching random walk in random environments is given as well. The classical  results by Biggins (1977) has been extended to the random environment situation.

\mbox{}\textbf{Keywords:}\quad smoothing transform,  functional equation,  random environment,  branching random walk,  martingales; \\
\mbox{}\textbf{Mathematics Subject Classification}:  Primary 60J80;
secondary 60G42.

\end{minipage}
\end{center}

\section{ Introduction and main results}

A random environment is modeled as an independent and identically distributed sequence of random variables, $\xi=(\xi_{n})_{n\in\mathbb{N}}$, indexed by the time $n\in\mathbb{N}=\{0,1,2,\cdots\},$ taking values in some measurable probability space $(\Theta,\mathcal{E})$. Without loss of generality we can suppose that $\xi$ is defined on the product space $(\Theta^{\mathbb{N}},\mathcal{E}^{\otimes\mathbb{N}},\tau)$, where $\tau$ is the law of $\xi$. Each realization of $\xi_{n}$ corresponds to a distribution $\eta_{n}=\eta(\xi_{n})$ on $[0,\infty)$. When the environment $\xi=(\xi_{n})_{n\in\mathbb{N}}$ is given, at each time $n$, there exists a random sequence of non-negative numbers that are ultimately zero, $\bar{Y}^{(n)}=(y_{1}^{(n)},y_{2}^{(n)},\cdots)$, of the distribution $\eta_{n}=\eta(\xi_{n})$.

For each realization $\xi\in\Theta^{\mathbb{N}}$ of the environment sequence, let $(\Gamma,\mathcal{G}, P_{\xi})$ be the probability space under which the process is defined. The probability $P_{\xi}$ is usually called quenched law. The total probability space can be formulated as the product space $(\Theta^{\mathbb{N}}\times\Gamma,\mathcal{E}^{\mathbb{N}}\otimes\mathcal{G},\mathbb{P})$, where $\mathbb{P}$ is defined that for all measurable and positive function $g$, we have $\int g \mathrm{d}\mathbb{P}=\int_{\Theta^\mathbb{N}}(\int_{\Gamma}g(\xi,y)\mathrm{d}P_{\xi}(y))\mathrm{d}\tau(\xi).$ The total probability $\mathbb{P}$ is usually called annealed law. Let $\mathbb{E}$ denote the expectation with respect to $\mathbb{P}$ and $E_{\xi}$ denote the expectation with respect to $P_{\xi}$.

Let $T$ denote the shift operator, given by if $\xi=(\xi_{0},\xi_{1},\cdots)$, then $T\xi=(\xi_{1},\xi_{2},\cdots)$. We are interested in the existence of solutions of the equation in distribution,
\begin{eqnarray}\label{samed}
Z(\xi)\overset{d}{=}\sum_{i\in\mathbb{N}_{+}}y_{i}^{(0)}(\xi)Z_{i}(T\xi)~~\text{ for a.e. } \xi ,
\end{eqnarray}
where $\overset{d}{=}$ denotes equality in distribution. When given the environment $\xi$, $Z_{i}(T\xi)~(i\in\mathbb{N}_{+})$ are independent and identically distributed non-negative random variables, and distributed the same as $Z(\xi)$, which is called the fixed point of the smoothing transform in random environment. The distribution of $Z_{i}(T\xi)~(i\in\mathbb{N}_{+})$ is determined by the environment $T\xi$, since $y_{i}^{(0)}(\xi)$ is determined by $\xi_{0}$, we have $Z_{i}(T\xi)~(i\in\mathbb{N}_{+})$ are independent with
$\bar{Y}^{(0)}(\xi)=(y_{1}^{(0)}(\xi),y_{2}^{(0)}(\xi),\cdots)$ and the distribution of $\sum_{i\in\mathbb{N}_{+}}y_{i}^{(0)}(\xi)Z_{i}(T\xi)$ is determined by $\xi$. Thus the question whether there exists $Z$ satisfies for a.e. $\xi$, $Z(\xi)\overset{d}{=}\sum_{i\in\mathbb{N}_{+}}y_{i}^{(0)}(\xi)Z_{i}(T\xi)$ is reasonable. For non-negative $Z$, (\ref{samed}) can be expressed naturally in terms of Laplace transform, it becomes the functional equation, for $u\geq0$,
\begin{eqnarray}\label{samel}
\phi(\xi,u)=E_{\xi}\left[\prod_{i\in\mathbb{N}_{+}}\phi\left(T\xi,uy_{i}^{(0)}(\xi)\right)\right]~~\text{ for a.e. } \xi ,
\end{eqnarray}
where $\phi(\xi,u)=E_{\xi}e^{-uZ(\xi)}$. If for $a.e.~\xi$, (\ref{samel}) is satisfied with some $\phi(\xi,u)$, then we call that (\ref{samel}) has a solution, and $Z(\xi)$ is said a solution of (\ref{samel}) or (\ref{samed}), i.e., the fixed points of the smoothing transform $H$, where the transform $H$ is defined as
$$ H\phi(T\xi,u)=E_{\xi}\left[\prod_{i\in\mathbb{N}_{+}}\phi\Big(T\xi,uy_{i}^{(0)}(\xi)\Big)\right].$$

The aim of this paper is to investigate the necessary and sufficient conditions under which  the solution with finite mean of (\ref{samel}) exits, and the uniqueness as well. Here the ``finite mean" is under the annealed probability. Specifically, for $c>0$, let $\mathcal{M}_{c}$ be the class of all probability measures on $[0,\infty)$ with the finite non-zero annealed mean $  c$, i.e.,
$$\mathcal{M}_{c}=\left\{\mu_{\xi}:  \text{for any  $\xi\in\Theta^{\mathbb{N}}$, $\mu_{\xi}$  is the probability distribution on $[0,\infty)$ with } \mathbb{E}\int x \mu_{\xi}(dx)=c   \right\},$$
in terms of Laplace transform, let
$$\mathcal{L}_{  c}=\left\{\phi(\xi,u)=E_{\xi}e^{-uX(\xi)}: \text{the probability distribution of $X$ belongs to } \mathcal{M}_{  c}, u\geq0, \xi\in\Theta^{\mathbb{N}}\right\}.$$ If for $a.e.~\xi$, (\ref{samel}) is satisfied with some $\phi\in\mathcal{L}_{  c}$, then we call that (\ref{samel}) has an $\mathcal{L}_{  c}$-solution, i.e., for $a.e.~\xi$,  $Z(\xi)$ is said an $\mathcal{L}_{  c}$-solution of  (\ref{samed}) iff it is a solution and $\mathbb{E} Z(\xi)=c.$\\

The fixed point of the smoothing transform has been investigated by many authors. Biggins (1977) got the necessary and sufficient conditions under which  the solution with finite mean  exits, and the uniqueness as well.
 Durrett and Liggett(1983) considered the  conditions under which  the solution (with possible infinite mean)  exits;  Liu (1998,2000) weakened the conditions of Durrett and Liggett(1983); Biggins and Kyprianou (1997,2005) further obtained the uniqueness in the boundary case; and some other related work, see for example, Alsmeyer, Biggins and Meiners (2012), Iksanov and Jurek (2002), Iksanov (2004), Caliebe and R\"{A}osler (2003) and Caliebe (2003), etc.

We will follow the line of Biggins (1977) to prove  and generalised the results (with finite mean) to the random environment situation ( However, it seems that there are some essential difficulties to generalise (by the analysis method) the (possible) infinite mean case (Durrett and Liggett(1983)) to random environment). To this end, we assume that for each $n\in\mathbb{N}$ and $\tau$-a.e. $\xi$,
\begin{eqnarray}\label{assume}
E_{\xi}\Big[\sum_{i\in\mathbb{N}_{+}}y_{i}^{(n)}(\xi)\Big]=1,~~~~P_{\xi}\Big[\sum_{i\in\mathbb{N}_{+}}\mathbbm{1}_{\{y_{i}^{(n)}(\xi)>0\}}<\infty\Big]=1,~~~~\mathbb{E}\log\Big[E_{\xi}\sum_{i\in\mathbb{N}_{+}}\mathbbm{1}_{\{y_{i}^{(n)}(\xi)>0\}}\Big]>0.
\end{eqnarray}
The conditions in (\ref{assume}) is  reasonable, which for example is satisfied for the branching random walk in random environment, see section \ref{application}.

In the following, we consider the model under condition (\ref{assume}).  Firstly, we have the sufficient conditions for the existence and uniqueness of the  solution of  (\ref{samel}).
\begin{theorem}\label{thm1}
If
\begin{eqnarray}
 &&\mathbb{E}\bigg[\sum_{i\in\mathbb{N}_{+}}y_{i}^{(0)}(\xi)\left(\log^{+}y_{i}^{(0)}(\xi)\right)^{2}\bigg]<\infty,\label{c1} \\
&&\mathbb{E}\bigg[\Big(\sum_{i\in\mathbb{N}_{+}}y_{i}^{(0)}(\xi)\Big)\Big|\log\Big(\sum_{i\in\mathbb{N}_{+}}y_{i}^{(0)}(\xi)\Big)\Big|\bigg]<\infty ,\label{c2}
\end{eqnarray}
and
\begin{eqnarray}\label{c3}
\mathbb{E}\bigg[\sum_{i\in\mathbb{N}_{+}}y_{i}^{(0)}(\xi)\log y_{i}^{(0)}(\xi)\bigg]<0.
\end{eqnarray}
then the equation for $a.e.~\xi$, $\phi(\xi,u)=H\phi(T\xi,u)$ has a unique solution in $\mathcal{L}_{1}$.
\end{theorem}
On the other hand, if one of the condition (\ref{c2}) or (\ref{c3}) dose not hold, we will show that no $\mathcal{L}_{1}$-solution exists.

\begin{theorem}\label{thm2}
When $\mathbb{E}\left[\sum_{i\in\mathbb{N}_{+}}y_{i}^{(0)}(\xi)\log y_{i}^{(0)}(\xi)\right]$ exists and is non-negative, the equation for $a.e.~\xi$, $\phi(\xi)=H\phi(T\xi)$ has no $\mathcal{L}_{1}$-solution.
\end{theorem}

\begin{theorem}\label{thm3}
When
\begin{eqnarray*}
 &&\mathbb{E}\bigg[\sum_{i\in\mathbb{N}_{+}}y_{i}^{(0)}(\xi)\left(\log^{+}y_{i}^{(0)}(\xi)\right)^{2}\bigg]<\infty,\\
&&-\infty<\mathbb{E}\bigg[\sum_{i\in\mathbb{N}_{+}}y_{i}^{(0)}(\xi)\log y_{i}^{(0)}(\xi)\bigg]<0,\\
\end{eqnarray*}
and
\begin{eqnarray*}
\mathbb{E}\bigg[\Big(\sum_{i\in\mathbb{N}_{+}}y_{i}^{(0)}(\xi)\Big)\Big|\log\Big(\sum_{i\in\mathbb{N}_{+}}y_{i}^{(0)}(\xi)\Big)\Big|\bigg]=\infty,
\end{eqnarray*}
then the equation for $a.e.~\xi$, $\phi(\xi,u)=H\phi(T\xi,u)$ has no $\mathcal{L}_{1}$-solution.
\end{theorem}

After 3 preparation Lemmas about the random environment have been specified, the proofs for the main results is given in section \ref{proof}; As an application,  the Biggins martingale convergence theorem for the branching random walk in a random environment is obtained in section \ref{application}.

\section{ Proofs of the main results}\label{proof}

Recall that $Z(\xi)$ is  an $\mathcal{L}_{  c}$-solution of  (\ref{samed}) if it is a solution  and  the annealed mean $\mathbb{E} Z(\xi)=c.$  Firstly, we will show that in this situation, the quenched mean is $c$ as well for $a.e.~\xi$, which is important in the following proofs for our results.
\begin{lemma}\label{l1}
Assume that $\phi(\xi,u)=E_{\xi}e^{-uZ(\xi)}$ is an $\mathcal{L}_{  c}$-solution to $(\ref{samel})$, then for $a.e.~\xi$, $E_{\xi}Z(\xi)=  c$.
\end{lemma}

\begin{proof}
Since $\phi$ is an $\mathcal{L}_{  c}$-solution to (\ref{samel}) we have
\begin{eqnarray}\label{le1}
\text{for $u\geq0,~~a.e.~\xi$,  }~~~~~ E_{\xi}e^{-uZ(\xi)}=E_{\xi}\left[\prod_{i\in\mathbb{N}_{+}}E_{T\xi}\left[e^{-uy_{i}^{(0)}(\xi)Z(T\xi)}\right]\right].
\end{eqnarray}

Let $A=\{\xi:E_{\xi}Z(\xi)\leq   c\}$, suppose that $Z_{i}(T\xi)(i\in\mathbb{N}_{+})$ are independent copies of $Z(T\xi)$. From (\ref{le1}) we have
\begin{eqnarray}
Z(\xi)\overset{d}{=}\sum_{i\in\mathbb{N}_{+}}y_{i}^{(0)}(\xi)Z_{i}(T\xi)~~\text{ for a.e. } \xi ,
\end{eqnarray}
combined with (\ref{assume}) we have
\begin{eqnarray*}
E_{\xi}Z(\xi)=E_{\xi}\sum_{i\in\mathbb{N}_{+}}y_{i}^{(0)}(\xi)Z_{i}(T{\xi})=E_{T\xi}Z(T\xi).
\end{eqnarray*}
Thus, $TA=A,~i.e.$ A is a $T$-invariant set, since $T$ is ergodic,
$\tau(A)=0\text{  or   }1.$
Since $\mathbb{E}Z=  c$, it can only be $\tau(\{\xi:E_{\xi}Z(\xi)\leq  c\})=1.$

In a similar way we can prove that $\tau(\{\xi:E_{\xi}Z(\xi)\geq  c\})=1.$
Thus $$\tau(\{\xi:E_{\xi}Z(\xi)=  c\})=1,$$ i.e. for $a.e.~\xi$, $E_{\xi}Z(\xi)=  c$.
\qed
\end{proof}

\begin{remark}
If $\phi(\xi,u)\in\mathcal{L}_{  c}(  c>0)$ is a solution of (\ref{samel}), then $\phi(\xi,\frac{u}{  c})\in\mathcal{L}_{1}$ and is also a solution of $(\ref{samel})$. That means if $(\ref{samel})$ has an $\mathcal{L}_{  c}$-solution($  c>0$), then $(\ref{samel})$ must have an $\mathcal{L}_{1}$-solution.
\end{remark}

\

Since for $a.e.~\xi$, $E_{\xi}\Big[\sum_{i\in\mathbb{N}_{+}}y_{i}^{(0)}(\xi)\Big]=1$ and $E_{\xi}\Big[\sum_{y_{i}^{(0)}(\xi)\leq y}y_{i}^{(0)}(\xi)\Big]\downarrow 0$ as $y\downarrow 0,$ then for $a.s.~\xi$, we can define the distribution function $G_{\xi}$ by the formula
$$G_{\xi}(\log y)=E_{\xi}\left[\sum_{y_{i}^{(0)}(\xi)\leq y}y_{i}^{(0)}(\xi)\right].$$
Since $y_{i}^{(0)}(T^{n}\xi)~(n\geq 0)$ is determined by $\xi_{n}$, the distribution function $G_{T^{n}\xi}$ is only determined by $\xi_{n}$. Let $X_{n}(\xi)~(n\in\mathbb{N})$ be independent random variables with the distribution function $G_{T^{n}\xi}$, then we have the following property.
\begin{lemma}
Under the annealed law, $(X_{n})_{n\in\mathbb{N}}$ are independent and identically distributed random variables.
\end{lemma}

\begin{proof}
From the definition of $X_{n}$ we have for any $n\geq0,~0<y<\infty$,
$$\mathbb{P}(X_{n}\leq\log y)=\mathbb{E}\big[G_{T^{n}\xi}(\log y)\big]=\mathbb{E}\big[G_{\xi}(\log y)\big]=\mathbb{E}\bigg[\sum_{y_{i}^{(0)}(\xi)\leq y}y_{i}^{(0)}(\xi)\bigg],$$
therefore $(X_{n})_{n\in\mathbb{N}}$ are identically distributed random variables. Next we justify the independence of $(X_{n})_{n\in\mathbb{N}}$. For any sequence $(A_{n})_{n\geq0}$ in $\mathcal{B}$,
\begin{eqnarray*}
\mathbb{P}(X_{n}\in A_{n},\forall n\geq0)&=&\mathbb{E}\big[P_{\xi}(X_{n}(\xi)\in A_{n},\forall n\geq0)\big]\\
&=&\mathbb{E}\bigg[\prod_{n\geq0}P_{\xi}(X_{n}(\xi)\in A_{n})\bigg]\\
&=&\prod_{n\geq0}\mathbb{E}\big[P_{\xi}(X_{n}(\xi)\in A_{n})\big]\\
&=&\prod_{n\geq0}\mathbb{P}(X_{n}\in A_{n}),
\end{eqnarray*}
the second equality comes from the fact that when given the environment $\xi$, $(X_{n}(\xi))_{n\in\mathbb{N}}$ are independent random variables, the third equality is due to the assumption that the environment $(\xi_{n})_{n\in\mathbb{N}}$ are $i.i.d.$ and the fact that the distribution of $X_{n}(\xi)$ is only determined by $\xi_{n}$.
\qed
\end{proof}

For $n\in\mathbb{N}_{+}$, let $S_{n}=\sum_{k=0}^{n-1}X_{k},~S_{0}=0$, since we have proved that $(X_{k})_{k\geq0}$ are $i.i.d$, $S_{n}$ is a random walk with
$$\frac{S_{n}}{n}\to\mathbb{E}[X_{0}]=\mathbb{E}\left[\sum_{i\in\mathbb{N}_{+}}y_{i}^{(0)}(\xi)\log y_{i}^{(0)}(\xi)\right],~~\mathbb{P}\text{-}a.e.$$

\begin{lemma}\label{youdong}
Assume that $\mathbb{E}\left[\sum_{i\in\mathbb{N}_{+}}y_{i}^{(0)}(\xi)\left(\log^{+}y_{i}^{(0)}(\xi)\right)^{2}\right]<\infty$, then $\sum_{n\in\mathbb{N}}\mathbb{P}[S_{n}\geq cn]<\infty$ for all $c>\mathbb{E}[X_{0}]$.
\end{lemma}

\begin{proof}
Note that in this case $(X_{k})_{k\geq0}$ are $i.i.d$, $$\mathbb{E}[(X_{0}^{+})^{2}]=\mathbb{E}\left[\sum_{i\in\mathbb{N}_{+}}y_{i}^{(0)}(\xi)\left(\log^{+}y_{i}^{(0)}(\xi)\right)^{2}\right]<\infty.$$ From Lemma 1 in \cite{BI77}, we get this result.
\qed
\end{proof}

\

We are now to prove the Theorems. The proofs are followed the line of Biggins (1977), with the above 3 preparation Lemmas about the random environment.

\begin{proof}\textbf{of Theorem~\ref{thm1} }
For $a.e.~\xi, u\geq 0$, let
\begin{eqnarray}\label{diedai}
\phi_{0}(\xi,u)=e^{-u},~~\phi_{n}(\xi,u)=H\phi_{n-1}(T\xi,u).
\end{eqnarray}
It can be easily checked that for each $n\in\mathbb{N}$, $\phi_{n}(\xi,u)$ is the quenched Laplace transform of some random variables under the environment $\xi$, thus we can assume that $\phi_{n}(\xi,u)=E_{\xi}e^{-uc_{n}(\xi)}$. From the iteration relationship (\ref{diedai}), we have
\begin{eqnarray*}
E_{\xi}e^{-uc_{n}(\xi)}=H\phi_{n-1}(T\xi,u)=E_{\xi}\Big[e^{-u\sum_{i\in\mathbb{N}_{+}}y_{i}^{(0)}(\xi)c_{n-1}^{i}(T\xi)}\Big],
\end{eqnarray*}
where $\big(c_{n-1}^{i}(T\xi)\big)_{(i\in\mathbb{N}_{+})}$ are independent copies of $c_{n-1}(T\xi)$. Since $\big(c_{n-1}^{i}(T\xi)\big)_{(i\in\mathbb{N}_{+})}$ are independent with $\big(y_{i}^{(0)}(\xi)\big)_{(i\in\mathbb{N}_{+})}$ and $E_{\xi}\big[\sum_{i\in\mathbb{N}_{+}}y_{i}^{(0)}(\xi)\big]=1$, we have for $n\geq1$,
$$E_{\xi}c_{n}(\xi)=E_{\xi}\bigg[\sum_{i\in\mathbb{N}_{+}}y_{i}^{(0)}(\xi)c_{n-1}^{i}(T\xi)\bigg]=E_{T\xi}c_{n-1}(T\xi),$$
iterating this gives $E_{\xi}c_{n}(\xi)=E_{T^{n}\xi}c_{0}(T^{n}\xi)=1$, where the last equality comes from the fact that $c_{0}\equiv1$.

Therefore when given the environment $\xi$, if $\{\phi_{n}(\xi,u)\}$ converges, the limit must be a Laplace transform under the environment $\xi$ (essentially because the constant mean implies tightness). What's more, if for $a.e.~\xi$, $\{\phi_{n}(\xi,u)\}$ converges, the limit is also a solution to (\ref{samel}); if for $a.e.~\xi$, the derivative of the limit at $u=0$ is one then this solution is in $\mathcal{L}_{1}$.

Let
\begin{equation*}
g_{n}(\xi,u)=u^{-1}|\phi_{n}(\xi,u)-\phi_{n-1}(\xi,u)|,
\end{equation*}
then
\begin{eqnarray*}
g_{n+1}(\xi,u)&=&u^{-1}\bigg|E_{\xi}\Big[\prod_{i\in\mathbb{N}_{+}}\phi_{n}\left(T\xi,uy_{i}^{(0)}(\xi)\right)\Big]-E_{\xi}\Big[\prod_{i\in\mathbb{N}_{+}}\phi_{n-1}\left(T\xi,uy_{i}^{(0)}(\xi)\right)\Big]\bigg|\\
&\leq&u^{-1}E_{\xi}\bigg|\prod_{i\in\mathbb{N}_{+}}\phi_{n}\left(T\xi,uy_{i}^{(0)}(\xi)\right)-\prod_{i\in\mathbb{N}_{+}}\phi_{n-1}\left(T\xi,uy_{i}^{(0)}(\xi)\right)\bigg|\\
&\leq&u^{-1}E_{\xi}\bigg[\sum_{i\in\mathbb{N}_{+}}\Big|\phi_{n}\left(T\xi,uy_{i}^{(0)}(\xi)\right)-\phi_{n-1}\left(T\xi,uy_{i}^{(0)}(\xi)\right)\Big|\bigg]\\
&=&E_{\xi}\Big[\sum_{i\in\mathbb{N}_{+}}y_{i}^{(0)}(\xi)g_{n}\left(T\xi,uy_{i}^{(0)}(\xi)\right)\Big]\\
&=&E_{\xi}\Big[g_{n}\left(T\xi,ue^{X_{0}(\xi)}\right)\Big].
\end{eqnarray*}
Iterating this gives
\begin{equation}\label{star}
g_{n+1}(\xi,u)\leq E_{\xi}\Big[g_{1}\left(T^{n}\xi,ue^{S_{n}(\xi)}\right)\Big].
\end{equation}
For $a.e.~\xi$, let
\begin{equation}\label{psi}
\psi(\xi,u)=\frac{\phi_{1}(\xi,u)-1+u}{u}.
\end{equation}
Then $\psi(\xi,u)$ is increasing in $u$ and for $a.e.~\xi$, $g_{1}(\xi,u)\leq\psi(\xi,u)$. Therefore, for any $c$, for $a.e.~\xi$,
\begin{equation}\label{gn}
\sum_{n=2}^{\infty}g_{n}(\xi,u)\leq \sum_{n=1}^{\infty}E_{\xi}\Big[\psi(T^{n}\xi,ue^{S_{n}(\xi)})\Big]\leq\sum_{n=1}^{\infty}P_{\xi}[S_{n}(\xi)\geq -cn]+\sum_{n=1}^{\infty}\psi(T^{n}\xi,ue^{-cn}),
\end{equation}
thus,
\begin{equation}
\mathbb{E}\sum_{n=2}^{\infty}g_{n}(\xi,u)\leq\mathbb{E}\sum_{n=1}^{\infty}P_{\xi}[S_{n}(\xi)\geq -cn]+\mathbb{E}\sum_{n=1}^{\infty}\psi(T^{n}\xi,ue^{-cn}).
\end{equation}
If we take $c$ satisfying the inequalities $0>-c>\mathbb{E}[X_{0}]=\mathbb{E}\left[\sum_{i\in\mathbb{N}_{+}}y_{i}^{(0)}(\xi)\log y_{i}^{(0)}(\xi)\right]$, then since $\mathbb{E}\left[\sum_{i\in\mathbb{N}_{+}}y_{i}^{(0)}(\xi)\left(\log^{+}y_{i}^{(0)}(\xi)\right)^{2}\right]<\infty$, Lemma \ref{youdong} shows that
$$\mathbb{E}\sum_{n=1}^{\infty}P_{\xi}[S_{n}(\xi)\geq -cn]=\sum_{n=1}^{\infty}\mathbb{P}[S_{n}(\xi)\geq -cn]<\infty.$$
Also since $\mathbb{E}\Big[\Big(\sum_{i\in\mathbb{N}_{+}}y_{i}^{(0)}(\xi)\Big)\Big|\log\Big(\sum_{i\in\mathbb{N}_{+}}y_{i}^{(0)}(\xi)\Big)\Big|\Big]<\infty$, $\mathbb{E}\Big[\sum_{i\in\mathbb{N}_{+}}y_{i}^{(0)}(\xi)\Big]=1$, from Doney ((1972), Lemma 3.4), we have
$$\mathbb{E}\sum_{n=1}^{\infty}\psi(T^{n}\xi,ue^{-cn})=\sum_{n=1}^{\infty}\mathbb{E}\psi(T^{n}\xi,ue^{-cn})=\sum_{n=1}^{\infty}\mathbb{E}\psi(\xi,ue^{-cn})<\infty.$$
As a result, we have $\mathbb{E}\sum_{n=2}^{\infty}g_{n}(\xi,u)<\infty$, that means for $a.e.~\xi$, $\sum_{n=2}^{\infty}g_{n}(\xi,u)<\infty$ and so for $a.e.~\xi$, $\lim_{n\to\infty}\phi_{n}(\xi,u)$ must exist. Furthermore for $a.e.~\xi$,
\begin{equation*}
u^{-1}|\lim_{n\to\infty}\phi_{n}(\xi,u)-e^{-u}|\leq \sum_{n=1}^{\infty}g_{n}(\xi,u)~~\text{and  } ~\psi(\xi,0+)=0,
\end{equation*}
therefore we may let $u\downarrow0$ in (\ref{gn}) to show that for $a.e.~\xi$, the derivative of $\lim_{n\to\infty}\phi_{n}(\xi,u)$ at $u=0$ is $1$, thus $\lim_{n\to\infty}\phi_{n}(\xi,u)$ is a solution of (\ref{samel}) in $\mathcal{L}_{1}$.

If $\phi$ and $\widetilde{\phi}$ are two $\mathcal{L}_{1}$-solutions to (\ref{samel}), for $a.e.~\xi$, let $$g(\xi,u)=u^{-1}|\phi(\xi,u)-\widetilde{\phi}(\xi,u)|;$$ then $g(\xi,0+)=0$ by Lemma \ref{l1}. As in (\ref{star}), $g(\xi,u)\leq E_{\xi}[g(T^{n}\xi,ue^{S_{n}(\xi)})]$ and since $e^{S_{n}(\xi)}\to 0$ almost surely we can see that for any $u>0,~a.e.~\xi,~g(\xi,u)=0.$ That means the equation has a unique solution in $\mathcal{L}_{1}$.

\qed
\end{proof}

\

On the other hand, if one of the condition (\ref{c2}) or (\ref{c3}) dose not hold, we will show that no $\mathcal{L}_{1}$-solution exists. To this end,
for $a.e.~\xi$, define the function $A(\xi,u)$ by the formula
\begin{equation}\label{Au}
uA(\xi,u)=E_{\xi}\left[\sum_{i\in\mathbb{N}_{+}}\left(1-\phi\left(T\xi,uy_{i}^{(0)}(\xi)\right)\right)-\Big(1-\prod_{i\in\mathbb{N}_{+}}\phi\left(T\xi,uy_{i}^{(0)}(\xi)\right)\Big)\right];
\end{equation}
then when $u>0$, it is easy to check that $A(\xi,u)$ is continuous and strictly greater than zero. Let
\begin{equation}\label{00}
\phi^{*}(\xi,u)=\frac{1-\phi(\xi,u)}{u},
\end{equation}
then since $\phi(\xi,u)=E_{\xi}\Big[\prod_{i\in\mathbb{N}_{+}}\phi\Big(T\xi,uy_{i}^{(0)}(\xi)\Big)\Big]$, we have
\begin{eqnarray*}
\phi^{*}(\xi,u)+A(\xi,u)&=&E_{\xi}\left[\frac{\sum_{i\in\mathbb{N}_{+}}\Big(1-\phi\left(T\xi,uy_{i}^{(0)}(\xi)\right)\Big)}{u}\right]\\
&=&E_{\xi}\left[\sum_{i\in\mathbb{N}_{+}}\frac{y_{i}^{(0)}(\xi)\left(1-\phi\left(T\xi,uy_{i}^{(0)}(\xi)\right)\right)}{uy_{i}^{(0)}(\xi)}\right]\\
&=&E_{\xi}\phi^{*}\left(T\xi,ue^{X_{0}(\xi)}\right).
\end{eqnarray*}

Iterating this gives for $a.e.~\xi$ ($S_{0}=0$),
\begin{equation}\label{2.7}
\phi^{*}(\xi,u)+E_{\xi}\left[\sum_{n=0}^{\infty}A\left(T^{n}\xi,ue^{S_{n}(\xi)}\right)\right]=\lim_{n\to\infty}E_{\xi}\left[\phi^{*}\left(T^{n}\xi,ue^{S_{n}(\xi)}\right)\right].
\end{equation}

\

At first, we consider the situation when  condition  (\ref{c3}) dose not holds,

\begin{proof}\textbf{of Theorem~\ref{thm2} }
Since $\mathbb{E}[X_{0}]=\mathbb{E}\left[\sum_{i\in\mathbb{N}_{+}}y_{i}^{(0)}(\xi)\log y_{i}^{(0)}(\xi)\right]\geq 0$, the random walk $\{S_{n}\}$ is recurrent or transient to $+\infty$ with probability 1.

Whenever $\{S_{n}\}$ drifts to $+\infty$, since for $a.e.~\xi$,
$$\phi^{*}\left(T^{n}\xi,ue^{S_{n}(\xi)}\right)\leq \frac{1}{ue^{S_{n}(\xi)}},~~~~~\phi^{*}(\xi,u)\leq 1,$$
using the dominated convergence theorem we have
\begin{eqnarray*}
\lim_{n\to\infty}E_{\xi}\left[\phi^{*}\left(T^{n}\xi,ue^{S_{n}(\xi)}\right)\right]&=&E_{\xi}\lim_{n\to\infty}\left[\phi^{*}\left(T^{n}\xi,ue^{S_{n}(\xi)}\right)\right]\\
&\leq& E_{\xi}\lim_{n\to\infty}\frac{1}{ue^{S_{n}(\xi)}}=0.
\end{eqnarray*}
Thus the right side of the equation $(\ref{2.7})$ is $a.e.$ 0. However, the left side of this equation $(\ref{2.7})$  is strictly positive by Lemma \ref{l1}, a contradiction. As a consequence,  the equation (\ref{samel}) has no $\mathcal{L}_{1}$-solution in this case.

When $\{S_{n}\}$ is persistent, choose a closed interval $I\subset (0,\infty)$, for any $u\in(0,\infty),$ let
$$\tau_{0}(\xi)=0,~~~\tau_{i}(\xi)=\inf\big\{k>\tau_{i-1}(\xi),ue^{S_{n}(\xi)}\in I\big\}(i\in\mathbb{N}_{+}),$$
since $\{S_{n}\}$ is persistent, for any $i\in\mathbb{N},~a.e.~\xi$, we have $\tau_{i}(\xi)<\infty.$ Then from the properties of $A(\xi,u)$ we have for any $k\in\mathbb{N}_{+},~P_{\xi}$-$a.e.,$
\begin{eqnarray}\label{a1}
\sum_{n=0}^{\infty}A\left(T^{n}\xi,ue^{S_{n}(\xi)}\right)&\geq&\sum_{i=1}^{k}A\left(T^{\tau_{i}(\xi)}\xi,ue^{S_{\tau_{i}(\xi)}(\xi)}\right)\nonumber\\
&\geq&\sum_{i=1}^{k}\inf_{j\in I}A\left(T^{\tau_{i}(\xi)}\xi,j\right)
\end{eqnarray}
when $k$ goes to infinity we have for $a.e.~\xi$,
\begin{equation}\label{a2}
\frac{\sum_{i=1}^{k}\inf_{j\in I}A(T^{\tau_{i}(\xi)}\xi,j)}{k}\to\mathbb{E}\inf_{j\in I}A(\xi,j)~~~~P_{\xi}\text{-}a.e.
\end{equation}
Since $I$ is a closed interval and $A(\xi,j)$ is strictly greater than zero when $j>0$, we have for $a.e.~\xi,~\inf_{j\in I}A(\xi,j)>0$, then $\mathbb{E}\inf_{j\in I}A(\xi,j)>0$. Combined (\ref{a1}) and (\ref{a2}) gives for $a.e.~\xi$, $P_{\xi}$-$a.e.,~\sum_{n=0}^{\infty}A\left(T^{n}\xi,ue^{S_{n}(\xi)}\right)=\infty,$ therefore for $a.e.~\xi$,
\begin{equation*}
E_{\xi}\left[\sum_{n=0}^{\infty}A\left(T^{n}\xi,ue^{S_{n}(\xi)}\right)\right]=\infty,
\end{equation*}
that is to say for $a.e.~\xi$, the left side of the equation (\ref{2.7}) is infinity, however
the other terms of equation $(\ref{2.7})$
  are finite  by Lemma \ref{l1}, this gives a contradiction. Thus, in this case, the equation (\ref{samel}) also has no $\mathcal{L}_{1}$-solution.
\qed
\end{proof}

\

When the condition (\ref{c2}) dose not hold, a little bit more attention should be paid to the random environment in the following proofs.

\begin{proof}\textbf{of Theorem~\ref{thm3} }
Since $\mathbb{E}[X_{0}]<0$ we know that for $a.e.~\xi$, $S_{n}(\xi)\to-\infty~P_{\xi}$-$a.e.$ and so the equation (\ref{2.7}) becomes for $a.e.~\xi$,
\begin{equation}\label{new}
E_{\xi}\left[\sum_{n=0}^{\infty}A\left(T^{n}\xi,ue^{S_{n}(\xi)}\right)\right]=1-\phi^{*}(\xi,u).
\end{equation}
For $a.e.~\xi$, let $e^{-\beta(\xi)}=\phi(\xi,1)$, then $\beta(\xi)\leq 1$ and the convex function $e^{\beta(\xi)u}\phi(\xi,u)$ is one at $u=0$ and $u=1$; therefore we have the inequalities for $a.e.~\xi$,
\begin{equation}\label{beta}
\phi(\xi,u)\leq e^{-\beta(\xi)u}~~~0\leq u\leq 1,\quad\quad\quad \phi(\xi,u)\geq e^{-\beta(\xi)u}~~~1\leq u<\infty.
\end{equation}
Then from (\ref{psi}) and (\ref{Au}) we have
\begin{eqnarray}\label{chazhi}
\begin{aligned}
&u\Big(\beta(T\xi)\psi\big(\xi,\beta(T\xi)u\big)-A(\xi,u)\Big)\\
&=E_{\xi}\bigg[\prod_{i\in\mathbb{N}_{+}}e^{-\beta(T\xi)uy_{i}^{(0)}(\xi)}\bigg]-1+\beta(T\xi)u\\
&-E_{\xi}\bigg[\sum_{i\in\mathbb{N}_{+}}
\left(1-\phi\left(T\xi,uy_{i}^{(0)}(\xi)\right)\right)-\bigg(1-\prod_{i\in\mathbb{N}_{+}}\phi\left(T\xi,uy_{i}^{(0)}(\xi)\right)\bigg)\bigg]\\
&=E_{\xi}\bigg[\prod_{i\in\mathbb{N}_{+}}e^{-\beta(T\xi)uy_{i}^{(0)}(\xi)}-\prod_{i\in\mathbb{N}_{+}}\phi\left(T\xi,uy_{i}^{(0)}(\xi)\right)-\sum_{i\in\mathbb{N}_{+}}\left(e^{-\beta(T\xi)uy_{i}^{(0)}(\xi)}-\phi\left(T\xi,uy_{i}^{(0)}(\xi)\right)\right)\\
&+\sum_{i\in\mathbb{N}_{+}}\left(e^{-\beta(T\xi)uy_{i}^{(0)}(\xi)}+\beta(T\xi)uy_{i}^{(0)}(\xi)-1\right)\bigg]
\end{aligned}
\end{eqnarray}

For $a.e.~\xi$, let
$$I(\xi)=\left\{i:y_{i}^{(0)}(\xi)>\frac{1}{u}\right\},\quad\quad\quad I(\xi)^{c}=\left\{i:y_{i}^{(0)}(\xi)\leq\frac{1}{u}\right\}.$$
Then from the inequalities $|\prod_{i}\alpha_{i}-\prod_{\ell}\beta_{\ell}|\leq \sum_{\ell}|\alpha_{\ell}-\beta_{\ell}|(|\alpha_{i}|,|\beta_{\ell}|\leq1)$ and (\ref{beta}),
\begin{eqnarray}
&&\prod_{i\in\mathbb{N}_{+}}e^{\beta(T\xi)uy_{i}^{(0)}(\xi)}-\prod_{i\in\mathbb{N}_{+}}\phi\left(T\xi,uy_{i}^{(0)}(\xi)\right)\nonumber\\
&\leq&\sum_{i\in I(\xi)^{c}}\left(e^{-\beta(T\xi)uy_{i}^{(0)}(\xi)}-\phi\left(T\xi,uy_{i}^{(0)}(\xi)\right)\right)
+\sum_{i\in I(\xi)}\left(\phi\left(T\xi,uy_{i}^{(0)}(\xi)\right)-e^{-\beta(T\xi)uy_{i}^{(0)}(\xi)}\right)\nonumber\\
\end{eqnarray}
Therefore the formula (\ref{chazhi}) yields
\begin{eqnarray*}
&&u\Big(\beta(T\xi)\psi\big(\xi,\beta(T\xi)u\big)-A(\xi,u)\Big)\\
&\leq&E_{\xi}\bigg[2\sum_{i\in I(\xi)}\left(\phi\left(T\xi,uy_{i}^{(0)}(\xi)\right)-e^{-\beta(T\xi)uy_{i}^{(0)}(\xi)}\right)+
\sum_{i\in\mathbb{N}_{+}}\left(e^{-\beta(T\xi)uy_{i}^{(0)}(\xi)}+\beta(T\xi)uy_{i}^{(0)}(\xi)-1\right)\bigg]\\
&\leq& 2E_{\xi}[\# I(\xi)]+E_{\xi}\bigg[\sum_{i\in\mathbb{N}_{+}}\left(e^{-\beta(T\xi)uy_{i}^{(0)}(\xi)}+\beta(T\xi)uy_{i}^{(0)}(\xi)-1\right)\bigg],
\end{eqnarray*}
the last inequality comes from the fact that $M_{1}(T\xi):=\sup\{\phi(T\xi,u)-e^{-\beta(T\xi)u}:u\geq1\}\leq1.$

Note that when $0<u<\frac{1}{e}$, since $\mathbb{E}\Big[\sum_{i\in\mathbb{N}_{+}}y_{i}^{(0)}(\xi)\big(\log^{+}y_{i}^{(0)}(\xi)\big)^{2}\Big]<\infty$, there exists a constant $M_{2}>1$ satisfies $\mathbb{E}\big[\frac{1}{u}(\log\frac{1}{u})^{2}E_{\xi}[\#I(\xi)]\big]\leq M_{2}$, thus
\begin{equation}\label{I1}
\mathbb{E}[\#I(\xi)]\leq uM_{2}(\log u)^{-2}.
\end{equation}
When $u\geq\frac{1}{e}$, since $E_{\xi}\Big[\sum_{i\in\mathbb{N}_{+}}y_{i}^{(0)}(\xi)\Big]=1$ from our assumption, then $E_{\xi}[\#I(\xi)]\frac{1}{u}\leq 1$. Therefore for $a.e.~\xi$,
\begin{equation}\label{I2}
E_{\xi}[\#I(\xi)]\leq u.
\end{equation}
Let
\begin{equation*}
h(u)=
\begin{cases}
(\log u)^{-2}, & 0<u<e^{-1},\\
1, & u\geq e^{-1},
\end{cases}
\end{equation*}
combined (\ref{I1}) and (\ref{I2}) we get
$$\mathbb{E}[\#I(\xi)]\leq M_{2}uh(u).$$
Also since $e^{-u}+u-1\leq u$ and $e^{-u}+u-1\leq u^{2}$, it is easy to establish that $e^{-u}+u-1\leq M_{3}uh(u)$ for some $M_{3}$. Therefore if we set $M=\max\{2M_{2},M_{3}\},$ then $\mathbb{E}\big(\beta(T\xi)\psi(\xi,\beta(T\xi)u)-A(\xi,u)\big)\leq M\Big[h(u)+\mathbb{E}\Big[\sum_{i\in\mathbb{N}_{+}}y_{i}^{(0)}(\xi)h(uy_{i}^{(0)}(\xi))\Big]\Big]$, and rewriting this
\begin{equation}\label{EA}
\mathbb{E}A(\xi,u)\geq\mathbb{E}\Big[\beta(T\xi)\psi(\xi,\beta(T\xi)u)\Big]-M\Big[h(u)+\mathbb{E}\big[h\big(ue^{\widetilde{X}_{0}(\xi)}\big)\big]\Big],
\end{equation}
where when given the environment $\xi,~\widetilde{X}_{0}(\xi)$ is a random variable which is independent of $(X_{n}(\xi))_{n\in\mathbb{N}}$ and has the same distribution function with $X_{0}(\xi)$.

Combing the inequality (\ref{EA}) with the equation (\ref{new}) yields
\begin{eqnarray*}
&&1-\mathbb{E}\phi^{*}(\xi,u)\\
&=&\mathbb{E}\Big[\sum_{n=0}^{\infty}A\big(T^{n}\xi,ue^{S_{n}(\xi)}\big)\Big]=\sum_{n=0}^{\infty}\mathbb{E}\Big[A\big(T^{n}\xi,ue^{S_{n}(\xi)}\big)\Big]\\
&\geq&\sum_{n=0}^{\infty}\left[\mathbb{E}\left(\beta(T^{n+1}\xi)\psi\big(T^{n}\xi,\beta(T^{n+1}\xi)ue^{S_{n}(\xi)}\big)\right)-\mathbb{E}\left(M h\big(ue^{S_{n}(\xi)}\big)+M h\big(ue^{\widetilde{S}_{n}(\xi)}\big)\right)\right],
\end{eqnarray*}
where $\widetilde{S}_{n}(\xi)=S_{n}(\xi)+\widetilde{X}_{0}(\xi)$. Let $\kappa=\mathbb{E}[X_{0}]=\mathbb{E}\Big[\sum_{i\in\mathbb{N}_{+}}y_{i}^{(0)}(\xi)\log y_{i}^{(0)}(\xi)\Big]$, then we can choose $\epsilon>0$ such that $\kappa+\epsilon<0$. let $\ell=\min\{n:n(\kappa+\epsilon)<-1\}$, combined with Lemma \ref{youdong}, we have
\begin{eqnarray*}
\sum_{n=0}^{\infty}\mathbb{E}h\big(e^{S_{n}})&\leq&\sum_{n=0}^{\infty}\mathbb{E}\Big[h(e^{S_{n}})\mathbbm{1}_{\{S_{n}\geq n(\kappa+\epsilon)\}}\Big]+\sum_{n=0}^{\infty}\mathbb{E}\Big[h(e^{S_{n}})\mathbbm{1}_{\{S_{n}< n(\kappa+\epsilon)\}}\Big]\\
&\leq&\sum_{n=0}^{\infty}\mathbb{P}[S_{n}\geq n(\kappa+\epsilon)]+\ell+\sum_{n=\ell}^{\infty}\frac{1}{n^{2}(\kappa+\epsilon)^{2}}<\infty.
\end{eqnarray*}
thus $\sum_{n=0}^{\infty}M\mathbb{E} h\big(e^{S_{n}(\xi)}\big)<\infty$, similarly we have $\sum_{n=0}^{\infty}M\mathbb{E} h\big(e^{\widetilde{S}_{n}(\xi)}\big)<\infty$. Since $$\sum_{n=0}^{\infty}\mathbb{E}\left[Mh\big(e^{S_{n}(\xi)})+Mh\big(e^{\widetilde{S}_{n}(\xi)}\big)\right]<\infty,$$ we have
\begin{eqnarray}\label{1-phi}
&&1-\mathbb{E}\phi^{*}(\xi,1)\\
&\geq&\sum_{n=0}^{\infty}\mathbb{E}\left[\beta(T^{n+1}\xi)\psi\big(T^{n}\xi,\beta(T^{n+1}\xi)e^{S_{n}(\xi)}\big)\right]-\sum_{n=0}^{\infty}\mathbb{E}\left[Mh\big(e^{S_{n}(\xi)})+Mh\big(e^{\widetilde{S}_{n}(\xi)}\big)\right].\nonumber
\end{eqnarray}
On the other hand, from Egorov Theorem there exists $n(\epsilon)<\infty$ satisfies for any $n\geq n(\epsilon)$, $\mathbb{P}\big(S_{n}\geq n(\kappa-\epsilon)\big)\geq c>0.$ Thus
\begin{eqnarray*}
&&\sum_{n=n(\epsilon)}^{\infty}\mathbb{E}\left[\beta(T^{n+1}\xi)\psi\big(T^{n}\xi,\beta(T^{n+1}\xi)e^{S_{n}(\xi)}\big)\right]\\
&\geq&\sum_{n=n(\epsilon)}^{\infty}\mathbb{E}\Big[E_{\xi}\Big[\beta(T^{n+1}\xi)\psi\big(T^{n}\xi,\beta(T^{n+1}\xi)e^{S_{n}(\xi)}\big)\Big|S_{n}(\xi)\geq n(\kappa-\epsilon)\Big]P_{\xi}\big[S_{n}(\xi)\geq n(\kappa-\epsilon)\big]\Big]\\
&\geq&\sum_{n=n(\epsilon)}^{\infty}\mathbb{E}\Big[E_{\xi}\Big[\beta(T^{n+1}\xi)\psi\big(T^{n}\xi,\beta(T^{n+1}\xi)e^{n(\kappa-\epsilon)}\big)\Big]P_{\xi}\big[S_{n}(\xi)\geq n(\kappa-\epsilon)\big]\Big]\\
&=&\sum_{n=n(\epsilon)}^{\infty}\mathbb{E}\Big[\beta(T^{n+1}\xi)\psi\big(T^{n}\xi,\beta(T^{n+1}\xi)e^{n(\kappa-\epsilon)}\big)\Big]\mathbb{P}\Big[S_{n}(\xi)\geq n(\kappa-\epsilon)\Big],
\end{eqnarray*}
the last equality is due to the fact that the distribution of $S_{n}(\xi)$ is only determined by $\xi_{0},\xi_{1},\cdots,\xi_{n-1}$, which is independent with $T^{n}\xi$. Therefore,
\begin{eqnarray*}
&&\sum_{n=n(\epsilon)}^{\infty}\mathbb{E}\left[\beta(T^{n+1}\xi)\psi\big(T^{n}\xi,\beta(T^{n+1}\xi)e^{S_{n}(\xi)}\big)\right]\\
&\geq&c\sum_{n=n(\epsilon)}^{\infty}\mathbb{E}\Big[\beta(T\xi)\psi\big(\xi,\beta(T\xi)e^{n(\kappa-\epsilon)}\big)\Big]\\
&=&c\sum_{n=n(\epsilon)}^{\infty}\int_{\Theta\times A}\beta(T\xi)\psi\big(\xi,\beta(T\xi)e^{n(\kappa-\epsilon)}\big)\mathrm{d}\tau(\xi)+c\sum_{n=n(\epsilon)}^{\infty}\int_{\Theta\times A^{c}}\beta(T\xi)\psi\big(\xi,\beta(T\xi)e^{n(\kappa-\epsilon)}\big)\mathrm{d}\tau(\xi),
\end{eqnarray*}
where $A=\{T\xi:\beta(T\xi)>\delta\}$. Recall that $e^{-\beta(T\xi)}=\phi(T\xi,1)$, $\phi$ is the $\mathcal{L}_{1}$-solution of (\ref{samel}), then we can choose $0<\delta<1$ satisfies $\mathbb{P}(A)>0$. Since $\psi(\xi,u)$ is only determined by $\xi_{0}$ and $\beta(T\xi)$ is independent with $\xi_{0}$ we have
\begin{eqnarray*}
\sum_{n=n(\epsilon)}^{\infty}\int_{\Theta\times A}\beta(T\xi)\psi\big(\xi,\beta(T\xi)e^{n(\kappa-\epsilon)}\big)\mathrm{d}\tau(\xi)
&\geq&\sum_{n=n(\epsilon)}^{\infty}\int_{\Theta\times A}\delta\psi\big(\xi,\delta e^{n(\kappa-\epsilon)}\big)\mathrm{d}\tau(\xi)\\
&=&\sum_{n=n(\epsilon)}^{\infty}\delta\mathbb{E}\psi\big(\xi,\delta e^{n(\kappa-\epsilon)}\big)\mathbb{P}(A).
\end{eqnarray*}
Since $\mathbb{E}\Big[\sum_{i\in\mathbb{N}_{+}}y_{i}^{(0)}(\xi)\Big]=1$ and $\mathbb{E}\Big[\Big(\sum_{i\in\mathbb{N}_{+}}y_{i}^{(0)}(\xi)\Big)\Big|\log\Big(\sum_{i\in\mathbb{N}_{+}}y_{i}^{(0)}(\xi)\Big)\Big|\Big]=\infty$,
using Doney ((1972), Lemma 3.4) we have $\sum_{n=n(\epsilon)}^{\infty}\mathbb{E}\psi\big(\xi,\delta e^{n(\kappa-\epsilon)}\big)=\infty$. Thus
$$\sum_{n=0}^{\infty}\mathbb{E}\left[\beta(T^{n+1}\xi)\psi\big(T^{n}\xi,\beta(T^{n+1}\xi)e^{S_{n}(\xi)}\big)\right]=\infty.$$
Combined with (\ref{1-phi}) we get $1-\mathbb{E}\phi^{*}(\xi,1)=\infty$, but this is a contradiction regarding the definition of $\phi^{*}(\xi,1)$ in (\ref{00}) and the $\mathcal{L}_{1}$-solution.  The proof is finished.
\qed
\end{proof}

\section{Applications in the branching random walk in a random environment}\label{application}

We set $\omega=(\omega_{n},n\in\mathbb{N})$ to be an $i.i.d.$ sequence of random variables with the values in the space of the distributions on the set of point processes on the real line, the law of $\omega$ is given by $\nu$. Conditionally on this sequence the process can be described as follows. At time 0, an initial ancestor, who forms the zeroth generation, is created at the origin. His children form the first generation and their positions on the real line are described by the point process $Z^{1}$ on $\mathbb{R}^{1}$, where $Z^{1}$ is a random locally finite counting measure and the distribution of $Z^{1}$ is determined by $\eta(\omega_{0})$. The people in the $n^{th}$ generation give birth independently of one another and of the preceding generations to form the $(n+1)^{th}$ generation. The point process describing the displacements of the children of a person in the $n^{th}$ generation from this person's position has the same distribution which is determined by $\eta(\omega_{n})$. Let $\{z_{r}^{n}\}$ be an enumeration of the positions of the people in the $n^{th}$ generation, and $Z^{n}$ be the point process with the atoms $\{z_{r}^{n}\}$. Suppose that for $a.e.~\omega,~P_{\omega}\big[Z^{1}(\mathbb{R})<\infty\big]=1$ and $\mathbb{E}\log\big[E_{\xi}Z^{1}(\mathbb{R})\big]>0.$ Define
\begin{equation}
m_{\omega_{i}}(\theta)=E_{T^{i}\omega}\left[\sum_{r}e^{-\theta z_{r}^{1}(T^{i}\omega)}\right]=E_{T^{i}\omega}\left[\int e^{-\theta t}d Z^{1}(t)\right].
\end{equation}

Let $$A=\{\theta:\mathbb{E} m_{\omega_{0}}(\theta)<\infty\},$$ for the following context, we restrict $\theta\in\mathring{A}$ and for each $\theta\in\mathring{A}$, we assume that there exists $\delta(\theta)>0$ satisfies for $a.e.~\omega$, $m_{\omega_{0}}(\theta)>\delta(\theta)$(uniform ellipticity condition). Observe that when given the environment, for $a.e.~\omega$,
\begin{equation}
W_{n}(\omega,\theta):=\frac{\sum_{r}e^{-\theta z_{r}^{n}}}{E_{\omega}\big[\sum_{r}e^{-\theta z_{r}^{n}}\big]}=\frac{\sum_{r}e^{-\theta z_{r}^{n}}}{m_{\omega_{0}}(\theta)\cdots m_{\omega_{n-1}}(\theta)}
\end{equation}
is a martingale with respect to the $\sigma$-field $\mathcal{F}_{n}(\omega)$, where $\mathcal{F}_{n}(\omega)$ is the $\sigma$-field generated by $Z^{1},Z^{2},\cdots,Z^{n}$ and $\omega$. Thus for $a.e.~\omega$, $W_{n}(\omega,\theta)$ has an almost sure limit, $W(\omega,\theta)$, and by Fatou's lemma $E_{\omega}[W(\omega,\theta)]\leq 1.$

For $n\in\mathbb{N}$, let $y_{i}^{(n)}(\omega)=\frac{e^{-\theta z_{i}^{1}(T^{n}\omega)}}{m_{\omega_{n}}(\theta)}$,(for those $y_{i}^{(n)}(\omega)(i\in\mathbb{N}_{+})$ without definition, we suppose them equal to zero), then obviously for each $n\in\mathbb{N}$, $(y_{i}^{(n)}(\omega))_{i\in\mathbb{N}_{+}}$ is a point process satisfies $E_{\xi}\big[\sum_{i\in\mathbb{N}_{+}}y_{i}^{(n)}(\xi)\big]=1$, what's more, using the way similar as \cite{BI77} (page 26), we have
\begin{equation}\label{Wn}
W_{n}(\omega,\theta)\overset{d}{=}\sum_{i\in\mathbb{N}_{+}}y_{i}^{(0)}(\omega)_{i}W_{n-1}(T\omega,\theta),
\end{equation}
where when given $\mathcal{F}_{1}(\omega),~ \{_{i}W_{n-1}(T\omega,\theta)\}$ are independent copies of $W_{n-1}(T\omega,\theta)$, if we now let $n$ tend to infinity we see that
\begin{equation}
W(\omega,\theta)\overset{d}{=}\sum_{i\in\mathbb{N}_{+}}y_{i}^{(0)}(\omega)_{i}W(T\omega,\theta),
\end{equation}
where when given $\mathcal{F}_{1}(\omega),~ \{_{i}W(T\omega,\theta)\}$ are independent copies of $W(T\omega,\theta)$ and the quenched Laplace transform of $W(\omega,\theta)$, $\hat{\phi}(\omega,u)=E_{\omega}e^{-uW(\omega,\theta)}$ satisfies (\ref{samel}).

Whenever $\mathbb{E}[W(\theta)]=c>0$ then $\hat{\phi}(\omega,\frac{u}{c})$ is an $\mathcal{L}_{1}$-solution to (\ref{samel}). Thus whenever (\ref{samel}) fails to have an $\mathcal{L}_{1}$-solution we must have $\mathbb{E}W(\theta)=0$ and then for $a.e.~\omega$, $E_{\omega}W(\theta)=0$.

Since $e^{-uW_{n}(\omega,\theta)}$ is a bounded submartingale converging to $e^{-uW(\omega,\theta)}$, we know that $$E_{\omega}[e^{-uW_{n}(\omega,\theta)}]\longrightarrow E_{\omega}[e^{-uW(\omega,\theta)}]=\hat{\phi}(\omega,u).$$
If we let $\phi_{n}(\omega,u)=E_{\omega}\big[e^{-uW_{n}(\omega,\theta)}\big]$, then we have for $a.e.~\omega$, $\phi_{0}(\omega,u)=e^{-u}$, and from (\ref{Wn}) $\phi_{n+1}(\omega,u)=E_{\omega}\big[\prod_{i\in\mathbb{N}_{+}}\phi_{n}\big(T\omega,uy_{i}^{(0)}(\omega)\big)\big]=H\phi_{n}(T\omega,u).$ Thus this definition is consistent with that given at (\ref{diedai}). Then when the condition of Theorem \ref{thm1} holds, $\hat{\phi}(\omega,u)$ is an $\mathcal{L}_{1}$-solution to the equation (\ref{samel}), then $\mathbb{E}W(\omega,\theta)=1$.

Note that in this model,
\begin{eqnarray}
\mathbb{E}\bigg[\Big(\sum_{i\in\mathbb{N}_{+}}y_{i}^{(0)}\Big)\Big|\log\Big(\sum_{i\in\mathbb{N}_{+}}y_{i}^{(0)}\Big)\Big|\bigg]
&=&\mathbb{E}\bigg[\Big(\sum_{i\in\mathbb{N}_{+}}\frac{e^{-\theta z_{i}^{1}(\omega)}}{m_{\omega_{0}}(\theta)}\Big)\Big|\log\Big(\sum_{i\in\mathbb{N}_{+}}\frac{e^{-\theta z_{i}^{1}(\omega)}}{m_{\omega_{0}}(\theta)}\Big)\Big|\bigg]\nonumber\\
&=&\mathbb{E}W_{1}(\theta)\big|\log W_{1}(\theta)\big|,
\end{eqnarray}
\begin{eqnarray}
\mathbb{E}\bigg[\sum_{i\in\mathbb{N}_{+}}y_{i}^{(0)}\log y_{i}^{(0)}\bigg]
&=&\mathbb{E}\bigg[\sum_{i\in\mathbb{N}_{+}}\frac{e^{-\theta z_{i}^{1}(\omega)}}{m_{\omega_{0}}(\theta)}\log\frac{e^{-\theta z_{i}^{1}(\omega)}}{m_{\omega_{0}}(\theta)}\bigg]\nonumber\\
&=&\mathbb{E}\bigg[\sum_{i\in\mathbb{N}_{+}}\frac{e^{-\theta z_{i}^{1}(\omega)}}{m_{\omega_{0}}(\theta)}\Big(-\theta z_{i}^{1}(\omega)-\log m_{\omega_{0}}(\theta)\Big)\bigg]\nonumber\\
&=&-\int\Big[-\theta\frac{m_{\omega_{0}}^{'}(\theta)}{m_{\omega_{0}}(\theta)}+\log m_{\omega_{0}}(\theta)\Big]\eta(\mathrm{d}\omega),
\end{eqnarray}

A straightforward calculation shows that the uniform ellipticity condition ensure that
\begin{eqnarray}
\mathbb{E}\bigg[\sum_{i\in\mathbb{N}_{+}}y_{i}^{(0)}\Big(\log^{+} y_{i}^{(0)}\Big)^{2}\bigg]
=\mathbb{E}\bigg[\sum_{i\in\mathbb{N}_{+}}\frac{e^{-\theta z_{i}^{1}(\omega)}}{m_{\omega_{0}}(\theta)}\Big(\log^{+}\frac{e^{-\theta z_{i}^{1}(\omega)}}{m_{\omega_{0}}(\theta)}\Big)^{2}\bigg]<\infty
\end{eqnarray}
holds for all $\theta\in\mathring{A}$. Therefore combine Theorem \ref{thm1}, \ref{thm2}, \ref{thm3} we have the following theorem.

\begin{theorem}\label{brw}
For any $\theta\in\mathring{A},$ assume that $$\kappa=\int\Big[-\theta\frac{m_{\omega_{0}}^{'}(\theta)}{m_{\omega_{0}}(\theta)}+\log m_{\omega_{0}}(\theta)\Big]\nu(\mathrm{d}\omega)$$
exists. Then $$\mathbb{E}[W(\theta)]=1$$
if and only if
\begin{equation}\label{con}
\mathbb{E}[W_{1}(\theta)|\log W_{1}(\theta)|]<\infty~~~~~\text{and}~~~~~\kappa>0
\end{equation}
and $\mathbb{E}W(\theta)=0$ when either of the conditions in (\ref{con}) fails.
\end{theorem}

\begin{remark}
Here we get Theorem \ref{brw}, the Biggins martingale convergence theorem (in random environment) by using analytical method. For the technical reason, we require the the uniform ellipticity condition. Indeed this result has been given in \cite{BK04} and \cite{KD04} by using probabilistic method.
\end{remark}


\begin{thebibliography}{99}
\baselineskip=14pt



\bibitem{ABM12} Alsmeyer, G., Biggins, J.D. and Meiners, M. (2012). The functional equation of the smoothing transform. {\it Ann. Probab}. \textbf{40}, 2069-2105.


\bibitem{BI77} Biggins, J.D. (1977). Martingale convergence in the branching random walk. {\it J. Appl. Prob}. \textbf{14}, 25-37.

\bibitem{BK97} Biggins, J.D. and Kyprianou, A.E. (1997). Seneta-Heyde norming in the branching random walk. {\it Ann. Probab}. \textbf{25}, 337-360.

\bibitem{BK04} Biggins, J.D. and Kyprianou, A.E. (2004). Measure change in multitype branching. {\it Adv. Appl. Probab}. \textbf{36}, 544-581.

\bibitem{BK05} Biggins, J.D. and Kyprianou, A.E. (2005). Fixed points of the smoothing transform: The
boundary case. {\it Electron. J. Probab.} \textbf{10}, 609-631.

\bibitem{CA03} Caliebe, A. (2003). Symmetric fixed points of a smoothing transformation. {\it Adv. Appl. Probab}. \textbf{35}, 377-394.

\bibitem{CR03} Caliebe, A. and R$\ddot{o}$sler, U. (2003a). Fixed points with finite variance of a smoothing transform. {\it Stoc. Proc. Appl}. \textbf{107}, 105-129.

\bibitem{DR72} Doney, R.A. (1972). A limit theorem for a class of supercritical branching processes. {\it J. Appl. Probab}. \textbf{9}, 707-724.

\bibitem{DL83} Durrett, R. and Liggett, M. (1983). Fixed points of the smoothing transform. {\it Z. Wahrsch. verw. Gebiete}. \textbf{64}, 275-301.

\bibitem{HLL14} Huang, C., Liang, X. and Liu, Q. (2014). Branching random walks with random environments in time. {\it Front. Math. China}. \textbf{9}, 835-842.

\bibitem{IA04} Iksanov, A.M. (2004). Elementary fixed points of the BRW smoothing transforms with infinite number of summands. {\it Stoc. Proc. Appl}. \textbf{114}, 27-50.

\bibitem{IJ02} Iksanov, A.M. and Jurek, Z.J. (2002). On fixed points of Poisson shot noise transform. {\it Adv. Appl. Probab}. \textbf{34}, 798-825.

\bibitem{KP76} Kahane, J.P. and Peyri$\grave{e}$re, J. (1976). Sur certaines martingales de Benoit Mandelbrot. {\it Adv. Math}. \textbf{22}, 131-145.

\bibitem{KD04} Kuhlbusch, D. (2004). On weighted branching processes in random environment. {\it Stoc. Proc. Appl}. \textbf{109}, 113-144.

\bibitem{LQ98} Liu, Q. (1998). Fixed points of a generalized smoothing transform and applications to the branching processes. {\it Adv. Appl. Probab}. \textbf{30}, 85-112.

\bibitem{LQ00} Liu, Q. (2000). On generalized multiplicative cascades. {\it Stoc. Proc. Appl}. \textbf{86}, 263-286.

\bibitem{PA92} Pakes, A.G. (1992). On characterizations via mixed sums. {\it Austral. J. Statist.} \textbf{34}, 323-339.

\bibitem{RU92} R$\ddot{o}$sler, U. (1992). A fixed point theorem for distributions. {\it Stoc. Proc. Appl}. \textbf{42}, 195-214.

 \end{thebibliography}
\end{document}